\documentclass[reqno]{amsart}
\usepackage{a4wide,calrsfs}
\usepackage{amsmath,bbm}
\usepackage{color}
\usepackage{epsfig}
\usepackage{subfigure}
\usepackage{amssymb}
\usepackage{enumerate}
\usepackage{graphicx}
\usepackage{subfigure}
\usepackage{amsmath}
\usepackage{rotating}
\usepackage{stmaryrd}
\usepackage{upgreek}

\newtheorem{theorem}{Theorem}[section]
\newtheorem{proposition}{Proposition}[section]

\newtheorem{lemma}{Lemma}[section]

\newtheorem{conjeture}{Conjeture}[section]

\allowdisplaybreaks
\numberwithin{equation}{section}

\title[Three power-exponential inequalities]{
The proof of three power-exponential inequalities}

\author[Coronel]{An{\'\i}bal\ Coronel$^\dag$}
\author[Huancas]{Fernando Huancas$^\dag$}

\thanks{$^\dag$ GMA, Departamento de Ciencias B\'asicas,
Facultad de Ciencias, Universidad del B\'{\i}o-B\'{\i}o,
Campus Fernando May, Chill\'{a}n, Chile,
E-mail: {\tt acoronel@ubiobio.cl, fihuanca@gmail.com}}

\date{\today}

\begin{document}

\begin{abstract}
In this paper we prove  three power-exponential inequalities for 
positive real numbers.
In particular, we conclude that this
proofs give affirmatively answers to
three, until now, open problems (conjectures~4.4, 2.1 and 2.2)
posed by C{\^{\i}}rtoaje in the following two works:
 ``{\it J. Inequal. Pure Appl. Math.} 10, Article 21, 2009'' 
and ``{\it J.  Nonlinear Sci. Appl.} 4:2:130-137, 2011''.
Moreover, we present a new proof of the inequality
$a^{ra}+b^{rb}\ge a^{rb}+b^{ra}$ for all positive
real numbers $a$ and $b$ and $r\in [0,e]$.
In addition, three new conjectures are presented.
\end{abstract}

\keywords{power inequalities, exponential inequalities, power-exponential inequalities}

\maketitle

\section{Introduction}

The power-exponential functions have useful applications in mathematical analysis
and in other theories like Statistics \cite{dahmani2013},
Biology \cite{NMO:NMO1572,NMO:NMO1024}, Optimization \cite{Park20011},
Ordinary Differential equations \cite{bruno2012} and
Probability \cite{ECP2318}. In the recent years there is
a intensive research in this area, 
see for instance 
\cite{miyagui_2014,miyagui_2013,yinli,hisasue2012,cirtoaje2,cirtoaje,
corhuanc,manyama_2010,ladislav_2008,fenqi,fenqi_1998}
and the recent overview on general mathematical inequalities done by
Cerone and Dragomir~\cite{pietrodragomir}. 
Some problems look like very simple
but, are  difficult to solve. For instance, we have
the following two classical problems: find the solution
of the equation $ze^z=a$ and  the basic problem of comparing $a^b$ and $b^a$
for all positive real numbers $a$ and $b$.
The first problem is perhaps one of the most ancient and useful problems concerning
to power-exponential functions,
see for instance \cite{wright_1959_1,wright_1959_2,wright_1960}. 
It was introduced by Lambert in \cite{lambert} and
have been studied  by recognized mathematicians like
Euler, P\'olya, Szeg\"o and Knuth, see \cite{euler,polya,corless}.
The solution to the problem  have been inspirited the definition
of the well-known $W$-Lambert function, see \cite{hoorfar_2008}. For the solution
to the second problem, see the discussion
given in \cite{bullen,luowen} and more recently in \cite{fenqi}.
Moreover, in spite of its algebraic simplicity, booth problems
are the central topic of a large number of research papers in the last years
(see \cite{miyagui_2014,cirtoaje2,corhuanc} and references therein).
In particular, in this paper, we are interested in some inequalities
conjectured by C{\^{\i}}rtoaje in \cite{cirtoaje,cirtoaje_2011},
which are very close to the second problem. To be more specific,
we start by recalling that in \cite{zeikii} was introduced and
probed the following assertion: the inequality
$a^a+b^b\ge a^b+b^a$ holds for all positive real
numbers less than or equal to~1. After that,
C{\^{\i}}rtoaje \cite{cirtoaje} introduce, prove and conjecture
several results about inequalities
for power-exponential functions.
In particular, in \cite{cirtoaje},  was established that the inequality
\begin{eqnarray}
 a^{ra}+b^{rb}\ge a^{rb}+b^{ra},
 \label{eq:cirtioaje_main}
\end{eqnarray}
holds true for $r\in [0,e]$ and for either $a\ge b\ge 1/e$ or
$1/e\ge a\ge b>0 $. However, in \cite{cirtoaje},
C{\^{\i}}rtoaje  leaves as an open problem the proof of 
\eqref{eq:cirtioaje_main} for $1>a>1/e>b>0$. 
Moreover, in \cite{cirtoaje}  the following
conjectures were introduced
\begin{enumerate}
 \item[]{\it Conjecture 4.3}. If $a,b,c$ are positive real numbers, then
 $a^{2a}+b^{2b}+c^{2c}\ge a^{2b}+b^{2c}+c^{2a}$.
  \item[]{\it Conjecture 4.4}. Let $r$ be a positive real number. The
  inequality
 \begin{eqnarray}
 a^{ra}+b^{rb}+c^{rc}\ge a^{rb}+b^{rc}+c^{ra},
 \label{main:inequality}
\end{eqnarray}
holds true for all positive real numbers $a,b,c$ with $a\le b\le c$
if and only if $r\le e.$
\item[]{\it Conjecture 4.6}. Let $r$ be a positive real number. The
inequality $a^{ra}+b^{ra}\ge 2$ holds for all nonnegative real numbers
$a$ and $b$ if and only if $r\le 3.$
\item[]{\it Conjecture 4.7}. If $a$ and $b$ are nonnegative real numbers
such that $a+b=2$, then $a^{3b}+b^{3a}+2^{-4}(a-b)^4\ge 2.$
\item[]{\it Conjecture 4.8}. If $a$ and $b$ are nonnegative real numbers
such that $a+b=1$, then $a^{2b}+b^{2a}\le 1.$
\end{enumerate}
Afterwards, the analysis of \eqref{eq:cirtioaje_main} was completed 
by  Manyama in \cite{manyama_2010}.
Thereafter, of the C{\^{\i}}rtoaje conjectures, 
the  milestones of the history are the works
of Coronel and Huancas \cite{corhuanc},
Matej{\'{\i}}{\v{c}}ka \cite{ladislav_2009},
Yin-Li \cite{yinli} and
Hisasue\cite{hisasue2012} (see also the work of C{\^{\i}}rtoaje \cite{cirtoaje_2011}), 
where they proved the conjectures 4.3,  4.6, 4.7 and 4.8,
respectively.  Here, we should be 
comment that the proof of conjecture 4.4 is still open.
Subsequently,
in 2011 C{\^{\i}}rtoaje introduce a new proof of \eqref{eq:cirtioaje_main}
and present the following three new conjectures 
\begin{enumerate}
 \item[]{\it Conjecture 2.1}. If $a,b\in ]0,1]$ and $r\in ]0,e],$
 then  $2\sqrt{a^{ra}b^{rb}}\ge a^{rb}+b^{ra}$.
  \item[]{\it Conjecture 2.2}. If $a,b,c\in ]0,1]$,
 then  $3a^{a}b^{b}c^{c}\ge (abc)^a+(abc)^b+(abc)^c$.
  \item[]{\it Conjecture 5.1}. If $a,b$ are nonnegative real numbers
  satisfying $a+b=1$. If $k\ge 1,$
 then  $a^{(2b)^k}+b^{(2a)^k}\le 1$.
\end{enumerate}
Recently, Miyagi and Nishizawa \cite{miyagui_2014} has been proved the Conjecture 5.1. However,
the conjectures 2.1 and 2.2 are still open. Thus, the  main focus of this 
paper are the proofs of  conjectures  4.4, 2.1 and 2.2.

The main contribution of the present paper is the development of the proof
of the following three theorems:
\begin{theorem}
\label{teo:cirtoaje}
The inequality \eqref{eq:cirtioaje_main} holds, for all
positive real numbers $a,b$ and for all $r\in [0,e]$.
\end{theorem}
\begin{theorem}
\label{teo:conjecture4.4}
The inequality \eqref{main:inequality} holds, for all
positive real numbers $a,b,c$ and for all $r\in [0,e]$.
\end{theorem}
\begin{theorem}
\label{teo:conjecture2.1}
The inequality 
\begin{eqnarray}
 2\sqrt{a^{ra}b^{rb}}\ge a^{rb}+b^{ra}
 \label{eq:main_raices}
\end{eqnarray}
holds, for all
positive real numbers $a,b$ and for all $r\in [0,e]$.
\end{theorem}
\begin{theorem}
\label{teo:conjecture2.2}
Let $n\in\mathbb{N}$ and $x_i\in ]0,1]^n$.
Then, the inequality
 \begin{eqnarray}
 n\prod_{i=1}^{n}x_i^{x_i}\ge \sum_{i=1}^{n}\Big(\prod_{j=1}^{n}x_j\Big)^{x_i}
 \label{ineq:conjecture2.2}
\end{eqnarray}
holds.
\end{theorem}
\noindent
Note that the conjectures 4.4, 2.1 and 2.2
are solved by the Theorems~\ref{teo:conjecture4.4}, 
\ref{teo:conjecture2.1}, \ref{teo:conjecture2.2},
respectively.
Moreover, we develop a proof of Theorem~\ref{teo:cirtoaje}
which is an alternative proof of \eqref{eq:cirtioaje_main}
for all positive real numbers $a$, $b$ and
$r\in [0,e]$, which is distinct to the existing proofs given in 
\cite{manyama_2010,cirtoaje_2011}.

The rest of the paper is organized in two sections: In section~\ref{sec:main_res}
we present the proofs of Theorems~\ref{teo:cirtoaje}, \ref{teo:conjecture4.4},
\ref{teo:conjecture2.1} and \ref{teo:conjecture2.2}
and in  section~\ref{sec:extensions} we present some remarks
and three new conjectures.

\section{Proofs of main results}
\label{sec:main_res}

In this section we present the proofs of 
Theorems~\ref{teo:cirtoaje}, \ref{teo:conjecture4.4},
\ref{teo:conjecture2.1}, \ref{teo:conjecture2.2}.
Firstly, we recall a result of \cite{corhuanc}.
Then, we present the corresponding proofs.

\subsection{A preliminar result}
For completeness and self-contained structure of the proofs
of Theorems~\ref{teo:cirtoaje} and \ref{teo:conjecture4.4},
we  need the following result of \cite{corhuanc}.

\begin{proposition}
\label{main:prop}
Consider $s\in \mathbb{R}^+$ with $s\not=1$, $m\in \mathbb{R}^+$
and $f,g:\mathbb{R}^+\to\mathbb{R}$
defined as follows
\begin{eqnarray*}
f(t)=t^s-t-\gamma^s+\gamma,
\quad
\mbox{and}
\quad
g(t)=
\left\{
  \begin{array}{ll}
    e^{-\ln(t)/(t-1)}, & \hbox{$t\not\in\{0,1\}$,} \\
    e^{-1}, & \hbox{$t=1$,}\\
    0, & \hbox{$t=0$.}\\
  \end{array}
\right.
\end{eqnarray*}
Then, the following properties are satisfied
\begin{enumerate}
\item[(i)] $f(\gamma)=0$ and $f(0)=f(1)=-\gamma^s+\gamma.$
\item[(ii)] If $s>1$,  $f$ is  strictly increasing  on $]g(s),\infty[$
and  strictly decreasing  on $]0,g(s)[.$
\item[(iii)] If $s\in ]0,1[$,  $f$ is  strictly decreasing  on $]g(s),\infty[$
and  strictly increasing  on $]0,g(s)[.$
\item[(iv)]   $g$ is continuous on   $\mathbb{R}^+\cup\{0\}$
and strictly increasing  on $\mathbb{R}^+$. Furthermore $y=1$
is a horizontal asymptote of $y=g(t).$
\end{enumerate}
\end{proposition}

\subsection{Proof of Theorem~\ref{teo:cirtoaje}}

Without loss of generality, we assume that $a>b.$ Indeed,
we follow the proof \eqref{eq:cirtioaje_main}
by application of Proposition~\ref{main:prop} with $t=a^{rb}, \gamma=b^{rb}$ and
$s=a/b.$ Indeed, we distinguish three cases
\begin{itemize}
  \item[(a$_2$)] {\bf Case} $a>b>1$ \Big($t>\gamma>1$ and $s>1$\Big).
  By Proposition~\ref{main:prop}-(iv),
  we note that $g(s)<1$. Then, by
  the strictly increasing behavior of $f$
  (Proposition~\ref{main:prop}-(ii)) we deduce the inequality since:
 \begin{eqnarray*}
 t>\gamma>1>g(s),\quad s>1
 \qquad
 \Rightarrow
 \qquad
 f(t)=a^{ra}- a^{rb}-b^{ra}+b^{rb}>f(\gamma)=0.
 \end{eqnarray*}

  \item[(b$_2$)] {\bf Case} $a>1\ge b$ \Big($t>1\ge \gamma$ and $s>1$\Big).  For
   $\gamma\in [g(s),1]$, we follow the inequality
  by almost identical arguments to that used before in (i),
  since  $t>1\ge \gamma\ge g(s)$ and $s>1$. Otherwise,
  if  $\gamma\in [0,g(s)]$, we deduce that
  \begin{eqnarray*}
 f(t)=a^{ra}- a^{rb}-b^{ra}+b^{rb}>f(1)=f(0)>f(\gamma)=0,
 \end{eqnarray*}
 which implies the desired inequality.

  \item[(c$_2$)] {\bf Case} $1>a>b>0$ \Big($1>t>\gamma>0$ and $s>1$\Big).
 First, we define $h:[0,1]\to\mathbb{R}$ by
 the correspondence rule $h(t)=-rt\ln t$ for $t>0$
 and $h(0)=0$. The function $h$ is concave and has a maximum
 at $(1/e,r/e)$. Thus, we deduce that
 \begin{eqnarray}
 -rb\ln b<1, \quad\mbox{for all $b\in[0,1]$ and $r\in [0,e]$}.
 \label{eq:tlnt}
 \end{eqnarray}
 Secondary,
 by the Napier's inequality \cite{napier}:
 \begin{eqnarray}
0<b<a\qquad
\Rightarrow\qquad
\frac{1}{a}<\frac{\ln a-\ln b}{a-b}<\frac{1}{b}.
\label{eq:napier}
\end{eqnarray}
From \eqref{eq:tlnt} and \eqref{eq:napier} we have that
\begin{eqnarray*}
-rb\ln b\le 1\le \frac{1}{a}<\frac{\ln a-\ln b}{a-b},
\end{eqnarray*}
which implies $\gamma>g(s).$ The proof of this case is completed by application
of Proposition~\ref{main:prop}-(ii).
\end{itemize}
Hence, by (a$_2$), (b$_2$) and (c$_2$) we follow that
Theorem~\ref{teo:cirtoaje} is valid.

\subsection{Proof of Theorem~\ref{teo:conjecture4.4}}

The proof of this theorem is again 
developed by application of Proposition~\ref{main:prop}.
Firstly, we recall the notation of \cite{corhuanc}:
\begin{eqnarray*}
\mathbb{R}^3_+&=&\{(a,b,c)\in \mathbb{R}^3\quad/\quad a>0,\quad
b>0\quad\mbox{and}\quad c>0\}
\\
\mathbb{E}_1&=&\Big\{(a,b,c)\in \mathbb{R}^3_+\quad/\quad\mbox{ $a=b=c$
or $a=b\not=c$ or $a\not=b=c$ }\Big\},
\\
\mathbb{E}^+_a&=&\Big\{(a,b,c)\in \mathbb{R}^3_+\quad/\quad\mbox{ $a\ge 1$
and $a>\max\{b,c\}$ }\Big\},
\\
\mathbb{E}^-_a&=&\Big\{(a,b,c)\in \mathbb{R}^3_+\quad/\quad\mbox{
$1>a>\max\{b,c\}$ }\Big\},
\\
\mathbb{E}^+_b&=&\Big\{(a,b,c)\in \mathbb{R}^3_+\quad/\quad\mbox{ $b\ge 1$
and $b>\max\{a,c\}$ }\Big\},
\\
\mathbb{E}^-_b&=&\Big\{(a,b,c)\in \mathbb{R}^3_+\quad/\quad\mbox{
 $1>b>\max\{a,c\}$ }\Big\},
\\
\mathbb{E}^+_c&=&\Big\{(a,b,c)\in \mathbb{R}^3_+\quad/\quad\mbox{ $c\ge 1$
and $c>\max\{a,b\}$ }\Big\}
\quad\mbox{and}
\\
\mathbb{E}^-_c&=&\Big\{(a,b,c)\in \mathbb{R}^3_+\quad/\quad\mbox{
$1>c>\max\{a,b\}$ }\Big\}.
\end{eqnarray*}
The family
$\Big\{\mathbb{E}_1,
\mathbb{E}^+_a,\mathbb{E}^-_a,
\mathbb{E}^+_b,\mathbb{E}^-_b,
\mathbb{E}^+_c,\mathbb{E}^-_c\Big\}$
is a set partition of
$\mathbb{R}^3_+$.
Now, with this notation, we subdivide the proof in three parts:
\begin{itemize}
\item[(a$_3$)] {\bf Case $(a,b,c)\in \mathbb{E}_1$}.
This special case is a direct consequence of Theorem~\ref{teo:cirtoaje}.

\item[(b$_3$)] {\bf Case $(a,b,c)\in \mathbb{E}^+_a\cup \mathbb{E}^+_b\cup \mathbb{E}^+_c$}.
If  $(a,b,c)\in \mathbb{E}^+_a$, we apply the Theorem~\ref{teo:cirtoaje}
 and
Proposition~\ref{main:prop} as follows.
We select $t=a^{rb},\gamma=c^{rb}$ and $s=a/b,$
the monotonic behavior and properties of function $f$,
defined on Proposition~\ref{main:prop},  implies that
\begin{eqnarray}
a^{ra}+c^{rb}>a^{rb}+c^{ra},
\label{casoii:amax:1}
\end{eqnarray}
since $t>\gamma,$ $t>1$ and $s>1$. Indeed, the corresponding proof of
\eqref{casoii:amax:1} needs the distinction of  two cases:
$c\ge 1$ and $c<1$. 
If $c\ge 1$, then $\gamma>1$ and $\gamma\in\,\, ]g(s),\infty[$,
so $f$ is strictly increasing and $t>\gamma$ implies \eqref{casoii:amax:1}.
For $c<1$, we note that
$\gamma<1$ and $-\gamma^s+\gamma\ge 0$ since $s>1$ and
$1\in ]g(s),\infty[$,
then the assumption $t>1$ implies  that $f(t)>f(1)=-\gamma^s+\gamma\ge 0=f(\gamma)$
and  \eqref{casoii:amax:1} is again true for this subcase.
Moreover, for $(a,b,c)\in \mathbb{E}^+_a\subset\mathbb{R}^3_+,$
by Theorem~\ref{teo:cirtoaje},
we recall that the inequality
\begin{eqnarray}
c^{rc}+b^{rb}>b^{rc}+c^{rb},
\label{casoii:amax:2}
\end{eqnarray}
holds true for all $r\in [0,e]$.
Adding \eqref{casoii:amax:1} and \eqref{casoii:amax:2}
we deduce~\eqref{main:inequality}.

The proof  for $(a,b,c)\in \mathbb{E}^+_b\cup \mathbb{E}^+_c$ is similar to the
case $(a,b,c)\in \mathbb{E}^+_a$ and we omit the details. 
However, we comment that for
$(a,b,c)\in \mathbb{E}^+_b$ we choose $t=b^{rc},\gamma=c^{2c}$ and $s=b/c$ and
for $(a,b,c)\in \mathbb{E}^+_c$ we select $t=c^{ra},\gamma=b^{ra}$ and $s=c/a.$

\item[(c$_3$)] {\bf Case $(a,b,c)\in \mathbb{E}^-_a\cup \mathbb{E}^-_b\cup \mathbb{E}^-_c$}.
Without loss of generality, we
assume that $(a,b,c)\in \mathbb{E}^-_a$ is such that $c<b<a$, since
the proof for $b<c<a$ is similar. 
We note that $\Omega=[0,e]\times [0,1]$ can be partitioned in the two sets
\begin{eqnarray*}
 \Omega_1&=&\Big\{(r,c)\in\Omega\quad:\quad c\in [(r-1)r^{-1},1]\Big\}
 \quad\mbox{and}
 \\
   \Omega_2&=&\Big\{(r,c)\in\Omega\quad:\quad c\in [0,(r-1)r^{-1}]\Big\}.
\end{eqnarray*}
Now, we continue the proof by distinguish 
the following two subcases: $(r,c)\in\Omega_1$
and $(r,c)\in\Omega_2.$

For the subcase $(r,c)\in\Omega_1$,
we apply the function $f$ given on Proposition~\ref{main:prop}
with $t=b^{rc},$ $\gamma=c^{rc}$ and $s=a/c$  to prove
\begin{eqnarray}
b^{ra}+c^{rc}>b^{rc}+c^{ra}
\quad
\mbox{for $0<c<b<a<1$ and $(r,c)\in \Omega_1$.}
\label{casoiii:amax:1}
\end{eqnarray}
Indeed, we firstly note that
the function $m:[c,1]\to \mathbb{R}$ defined as follows 
$m(z)=zc^{rz}-c^{rc+1}$ has the following properties:
\begin{enumerate}
\item[(m$_a$)] $m(c)=0$; 
\item[(m$_b$)] $m(1)=c^r(1-c^{rc+1-r})\ge 0$ for all
$(r,c)\in\Omega_1$ since $c>(r-1)/r$; and
\item[(m$_c$)] $m$ has a maximum at $z_{\max}=-1/r\ln c$, since
the first and second derivatives of $m$ are given by 
$m'(z)=c^{rz}(1+rz\ln c)$ and $m''(z)=c^{rz}(2r+rz\ln c)\ln c$
and naturally $m'(z_{\max})=0$ and $m''(z_{\max})<0$.
\end{enumerate}
Moreover,
we notice that $z_{\max}\ge c$ is equivalently to $1>-rc\ln c$, which is true
for $r\in [0,e]$ and $c\in [0,1]$, see the proof of~\eqref{eq:tlnt}. Then,
by (m$_a$)-(m$_c$), we
follow that $m(z)\ge 0,$ for all $z\in [c,1]$. In particular, for $z=a$,
we have that
\begin{eqnarray}
ac^{ra}>c^{rc+1},
\quad
\mbox{for $a\in [c,1]\subset [0,1]$ and $r\in [0,e]$.}
\label{casoiii:amax:2}
\end{eqnarray}
Now, from \eqref{casoiii:amax:2}, we note that
\begin{eqnarray}
ac^{ra}>c^{rc+1}
&\Rightarrow &
c^{r(a-c)}>\frac{c}{a}
\quad\Rightarrow\quad
rc\ln c>\frac{c\ln(c/a)}{a-c}
\nonumber\\
&\Rightarrow&
c^{rc}>e^{\frac{-c\ln(a/c)}{a-c}}
\quad\Rightarrow\quad
\gamma>g(s),
\end{eqnarray}
which implies \eqref{casoiii:amax:1}
by application of  Proposition~\ref{main:prop}-(ii),
since $t>\gamma>g(s)$ and $f$ is increasing on $]g(s),\infty[$.

For the subcase $(r,c)\in\Omega_2$, 
we apply the function $f$ given on Proposition~\ref{main:prop}
with $t=b^{rc},$ $\gamma=c^{rc}$ and $s=a/c$  to prove
\begin{eqnarray}
b^{ra}+c^{rc}>b^{rc}+c^{ra}
\quad
\mbox{for $0<c<b<a<1$ and $(r,c)\in \Omega_2$.}
\label{casoiii:amax:1_w2}
\end{eqnarray}
We note that the inequality $c^{rc}>c^{r-1}$ holds true
for all $(r,c)\in \Omega_2.$ Now, in order to deduce that
$\gamma>g(s)$ is suficiently to prove that $c^{r-1}>g(s)$. Indeed, the 
function $q:[c,1]\to \mathbb{R}$ defined as follows
$q(z)=c^{(1-r)z}z^c-c^{c+c(1-r)}$ has the following properties:
\begin{enumerate}
\item[(q$_a$)] $q(c)=0$;
\item[(q$_b$)] $q(1)=c^{1-r}(1-c^{c+(c-1)(1-r)})\ge 0$ for 
all $(r,c)\in \Omega_2$, since $c\in [0,(r-1)/r]$; and
\item[(q$_c$)] $q$ is increasing in $[c,1].$
\end{enumerate}
Then,
we deduce that $q(z)\ge 0$ for all $z\in [c,1]$.
In particular, for $z=a\in [c,1]$, we deduce
that $c^{(1-r)a}a^c-c^{c+c(1-r)}\ge 0$, which implies the following sequence
of implications
\begin{eqnarray*}
c^{(1-r)a}a^c>c^{c+c(1-r)}
\quad
\Rightarrow
\quad
\frac{c^{(1-r)a}}{c^{c(1-r)}}>\frac{c^c}{a^c}
\quad
\Rightarrow
\quad
c^{1-r}>g(s).
\end{eqnarray*}
Thus \eqref{casoiii:amax:1_w2} holds true.

From  \eqref{casoiii:amax:1} and  \eqref{casoiii:amax:1_w2},
we deduce that
\begin{eqnarray}
b^{ra}+c^{rc}>b^{rc}+c^{ra}
\quad
\mbox{for $0<c<b<a<1$ and $r\in [0,e]$.}
\label{casoiii:amax:1_w3}
\end{eqnarray}
Hence, to complete the proof for $0<c<b<a<1$, 
we add the inequality \eqref{casoiii:amax:1_w3} with
$a^{ra}+b^{rb}>a^{rb}+b^{ra}$ for $r\in[0,e]$, 
which is true by Theorem~\ref{teo:cirtoaje}.

For $(a,b,c)\in \mathbb{E}^-_b\cup \mathbb{E}^-_c$ we
can follow line by line the proof of
$(a,b,c)\in \mathbb{E}^-_a$. However, we can obtain
a direct proof by apply the result obtained
for $(a,b,c)\in \mathbb{E}^-_a$ by interchanging the role of variables.
For instance, if $(a,b,c)\in \mathbb{E}^-_b$ then $(b,a,c)\in \mathbb{E}^-_a$
which implies \eqref{main:inequality}.

\end{itemize}
Hence, by (a$_3$), (b$_3$) and (c$_3$) we follow the complet proof
of Theorem~\ref{teo:conjecture4.4}.

\subsection{Proof of Theorem~\ref{teo:conjecture2.1}}
Given $b\in ]0,1]$, we define the function $H:]0,1]\to\mathbb{R}$ as follows
\begin{eqnarray*}
H(x)=2\sqrt{x^{rx}b^{rb}}-b^{rx}-x^{rb}.
\end{eqnarray*}
Then we prove that $H(x)>0$ for all $x\in ]0,1]$, which naturally implies
the inequality $2\sqrt{a^{ra}b^{rb}}\ge a^{rb}+b^{ra}$ for $x=a.$
Indeed, we prove that the function $H$ has a global
minimum at $x=b$. The fact that in $x=b$ there is a local minimum
of $H$, follows by noticing that $H'(b)=0$ and $H''(b)>0$, since
\begin{eqnarray*}
H'(x)&=&r\Bigg[\sqrt{x^{rx}b^{rb}}\Big(\ln x + 1\Big)-b^{rx}\ln b-bx^{rb-1}\Bigg]
\quad
\mbox{and}
\\
H''(x)&=&r\Bigg[\sqrt{x^{rx}b^{rb}}
\Bigg\{r\Big(\ln x + 1\Big)^2+x^{-1}\Bigg\}
-rb^{rx}(\ln b)^2-b(rb-1)x^{rb-2}\Bigg].
\end{eqnarray*}
Meanwhile, the property that $b$ is a global minimum of $H$ can be proved
by  rewriting $H'$ as the difference of two functions
and by analyzing the
sign of~$H'$ using some properties of this new functions. Indeed, to be more specific,
we note that $H'(x)=r[K(x)-Q(x)]$
for all $x\in ]0,1]$, where
the functions $K$ and $Q$ are defined as follows
\begin{eqnarray*}
K(x)=\sqrt{x^{rx}b^{rb}}\Big(\ln x + 1\Big)
\quad
\mbox{and}
\quad
Q(x)=b^{rx}\ln b+bx^{rb-1}.
\end{eqnarray*}
The functions $K$ and $Q$ have the following properties
\begin{itemize}
\item[(K$_1$)] $K$ is strictly increasing on $]0,1]$, since 
$K'(x)=\sqrt{x^{rx}b^{rb}}\Big\{r\Big(\ln x + 1\Big)^2+x^{-1}\Big\}>0,$
for all $x\in ]0,1].$
\item[(K$_2$)] $K(x)\to-\infty$ when $x\to 0^+$,
$K(1/e)=0$ and $K(1)=\sqrt{b^{rb}}$.
\item[(Q$_1$)] The derivative of $Q$ is given by 
$Q'(x)=rb^{rx}(\ln b)^2+b(rb-1)x^{rb-2},$
for all $x\in ]0,1].$ Then, in order to analize
the sign of $Q'$, we introduce the set $\Lambda=]0,1]\times]0,e]$
and a partition $\{\Lambda_1,\Lambda_2,\Lambda_3\}$ 
of~$\Lambda$, where
\begin{eqnarray*}
\Lambda_1&=&\Big\{(b,r)\in \Lambda\;:\; Q'(x)>0\;\mbox{ for all $x\in]0,1]$}\Big\}
\\
\Lambda_2&=&\Big\{(b,r)\in \Lambda\;:\; Q'(x)<0\;\mbox{ for all $x\in]0,1]$}\Big\}
\\
\Lambda_3&=&\Big\{(b,r)\in \Lambda\;:\; \exists! c\in]0,1]\;
	\mbox{ such that $Q$ has a minimum at $x=c$}\Big\}.
\end{eqnarray*}
We note that the sets $\Lambda_i$, $i=1,2,3$, are not empty
since for instance $]0,1[\times [1/b,e]\subset \Lambda_1$
for all $b\in ]0,1]$, $\{1\}\times]0,1[ \subset \Lambda_2$
and $]0,1[\times \{1\}\subset \Lambda_3$. Moreover, we note that $rb>1$ implies
that $(b,r)\in \Lambda_1$ and naturally $\Lambda_2\cup\Lambda_3$ is
a subset of $]0,1]\times ]0,1/b[$. The uniqueness of $c$ can be deduced by noticing
that the solution of $Q'(x)=0$
is equivalent to the intersection of the following two monotone functions
$S(x)=rb^{rx}(\ln b)^2$ and $J(x)=b(1-rb)x^{rb-2}.$

\item[(Q$_2$)] $Q(x)\to\ln b$ when $x\to 0^+$,
and $Q(1)=b^{r}\ln b+b$.
\end{itemize}
From (K$_1$) and (Q$_1$) we deduce the uniqueness of $b\in ]0,1]$
such that $Q(b)=K(b)$ or equivalently $H'(b)=0$. Now, 
from (K$_2$) and (Q$_2$), we note that
$Q(0^+)>K(0^+)$ for all $(r,b)\in \Lambda$ since $K(0^+)=-\infty$.
Then, $H'(x)<0$ for all $x\in ]0,b[$. 
Additionally, from (K$_2$) and (Q$_2$), we observe that $Q(1)<K(1)$.
This fact is a consequence of that the function 
$F(w,r)=\sqrt{w^{rw}}-w^r\ln (w)-w$ is strictly decreasing in
$r$, since $F_r(w,r)=\ln(w)\Big((r/2)\sqrt{w^{rw}}-w^r\ln (w)\Big)<0$.
Consequently, for $r<e$ we have that $F(w,r)>F(w,e)=\sqrt{w^{ew}}-w^e\ln (w)-w>0$
for all $w\in ]0,1].$ Hence, for $w=b$ we get that $F(b,r)>0$ or
$Q(1)<K(1)$, which implies that  $H'(x)>0$ for all $x\in ]b,1]$.
Thus, $b$ is a global minimum of $H$. Therefore,
$H(x)\ge H(b)=0$ for all $x\in ]0,1]$ and in particular
for $x=a.$

\subsection{Proof of Theorem~\ref{teo:conjecture2.2}}
The proof follows by the fact that the function
$P:]0,1]^{n-1}\to\mathbb{R}$ defined the following
correspondence rule
\begin{eqnarray*}
 P(z_1,\ldots,z_{n-1})=
 nx_n^{x_n}\prod_{i=1}^{n}z_i^{z_i}
 -\Bigg(x_n\prod_{j=1}^{n-1}z_j\Bigg)^{x_n}
 -\sum_{i=1}^{n-1}\Bigg(x_n\prod_{j=1}^{n-1}z_j\Bigg)^{x_i},
 \quad x_n\in ]0,1],
\end{eqnarray*}
has a global minimum at $(z_1,\ldots,z_{n-1})=(x_n,\ldots,x_n)$.
Indeed, by simplicity of notation we develop the details
of the proof for $n=3$ and with $(x_1,x_2,x_3)=(a,b,c)$.
Note that, in this case for an
arbitrary $c\in ]0,1]$, the function $P:]0,1]^2\to\mathbb{R}$ 
has the following form
\begin{eqnarray*}
 P(x,y)=3x^xy^yc^c-(xyc)^x-(xyc)^y-(xyc)^c.
\end{eqnarray*}
Then, we have that
\begin{eqnarray*}
 P_x(x,y)&=&3x^xy^yc^c\Big(\ln(x) +1\Big)
	-\Big(\ln(xyc) +1\Big)(xyc)^x
	-\frac{y}{x}(xyc)^y
	-\frac{c}{x}(xyc)^c,
\\
 P_y(x,y)&=&3x^xy^yc^c\Big(\ln(y) +1\Big)
	-\frac{x}{y}(xyc)^x
	-\Big(\ln(xyc) +1\Big)(xyc)^y
	-\frac{c}{y}(xyc)^c,
\\
 P_{xx}(x,y)&=&3x^xy^yc^c\left[\frac{1}{x}+\Big(\ln(x) +1\Big)^2\right]
\\
&&
	-\left[\frac{1}{x}+\Big(\ln(xyc) +1\Big)^2\right](xyc)^x
	-\left[\frac{y^2-y}{x^2}\right](xyc)^y
	-\left[\frac{c^2-c}{x^2}\right](xyc)^c,
\\
 P_{yy}(x,y)&=&3x^xy^yc^c\left[\frac{1}{y}+\Big(\ln(y) +1\Big)^2\right]
\\
&&
	-\left[\frac{x^2-x}{y^2}\right](xyc)^x
	-\left[\frac{1}{y}+\Big(\ln(xyc) +1\Big)^2\right](xyc)^y
	-\left[\frac{c^2-c}{y^2}\right](xyc)^c,
\\
 P_{xy}(x,y)&=&P_{yx}(x,y)=3x^xy^yc^c\Big(\ln(y) +1\Big)\Big(\ln(x) +1\Big)
\\
&&
	-\left[\frac{x\Big(\ln(xyc)+1\Big)+1}{y}\right](xyc)^x
	-\left[\frac{y\Big(\ln(xyc)+1\Big)+1}{x}\right](xyc)^y
	-\left[\frac{c^2}{xy}\right](xyc)^c.
\end{eqnarray*}
An evaluation at $(c,c)$ implies that
\begin{eqnarray*}
P_x(c,c)&=&P_y(c,c)=0,
\\
P_{xx}(c,c)&=&P_{yy}(c,c)=c^{3c-1}\Big(-6c(\ln(c))^2+4\Big),
\\
P_{xy}(c,c)&=&P_{yx}(c,c)=c^{3c-1}\Big(3c(\ln(c))^2-2\Big)
\end{eqnarray*}
Now, defining  $\mathbb{P}_1(w)=-6w(\ln(w))^2+4$ and
$\mathbb{P}_2(w)=27w^2(\ln w)^4-24w(\ln w)^2+4$,
we observe that  $P_{xx}(c,c)=c^{3c-1}\mathbb{P}_1(c)$ and
$P_{xx}(c,c)P_{yy}(c,c)-P_{xy}(c,c)P_{xy}(c,c)=c^{2(3c-1)}\mathbb{P}_2(c)$.
Then, the Hessian matrix asociated to $P$ at $(c,c)$
is positive semidefinite since both functions,
$\mathbb{P}_1$ and $\mathbb{P}_2$,
are positive on  $]0,1]$ or equivalently the function $P$ has
a local minimum at $(c,c)$. Now, we deduce that $(c,c)$
is the global minimum since we can prove that $(c,c)$ is the unique solution
of $(P_x,P_y)=(0,0)$. Indeed,  assuming that
there is $(x,y)$ with $x\not=y\not=c$  such that
$P_x(x,y)=P_y(x,y)=0$, we can deduce a contradiction. 
Note that
\begin{eqnarray*}
0&=&\Big|P_x(x,y)-P_y(x,y)\Big|
\\
&\ge &\Bigg|\min\Big\{3x^xy^yc^c,(xyc)^x,(xyc)^y,(xyc)^c\Big\}\Bigg|
\Bigg|\ln\Bigg(\frac{x}{y}\Bigg)-\frac{y}{x}-\frac{c}{x}+\frac{x}{y}+\frac{c}{y}\Bigg|
\\
&\ge &\Bigg|\min\Big\{3x^xy^yc^c,(xyc)^x,(xyc)^y,(xyc)^c\Big\}\Bigg|
\Bigg|\frac{1}{x}+\frac{x+y}{xy}+\frac{c}{xy}\Bigg|
\Big|x-y\Big|,
\end{eqnarray*}
since the inequality $\ln(r)> (r-1)/r$ holds for all $r>0$ and $r\not=1$
(see for instance \cite{bullen}). Then, $x=y$, which is a contradiction
with the assumption that $x\not=y.$ Thus, we have $(c,c)$ is a global
minimum of the function $P$ or equivalently
$P(x,y)\ge P(c,c)=0$ for all $(x,y)\in ]0,1]^2$,
which implies the desired inequality for $(x,y)=(a,b).$

\section{Aditional remarks on posible generalizations}
\label{sec:extensions}

In this section we present the posible extensions of Theorems~\ref{teo:cirtoaje},
\ref{teo:conjecture4.4} and \ref{teo:conjecture2.1} 
 to a sequence of  positive real numbers.
We note that
the natural generalizations of \eqref{main:inequality} and \eqref{eq:main_raices}
are given by
 \begin{eqnarray}
 &&\sum_{i=1}^{n}x_i^{rx_i}\ge x_n^{rx_1}+ \sum_{i=1}^{n-1}x_i^{rx_{i+1}},
 \quad (x_1,\ldots,x_n)\in \mathbb{R}^n_+, 
 \quad r\in [0,e],
 \label{main:inequality:n}
 \\
 &&n\sqrt[n]{\prod_{i=1}^{n}x_i^{rx_i}}\ge x_n^{rx_1}+ \sum_{i=1}^{n-1}x_i^{rx_{i+1}},
 \quad (x_1,\ldots,x_n)\in ]0,1]^n,
 \label{main:raices:n}
\end{eqnarray}
respectively. We present a partial  proof of \eqref{main:inequality:n}
(see Lemma~\ref{conj:main:n}, below) and  leaves as a conjecture
the proof of  \eqref{main:raices:n}.

\begin{lemma}
\label{conj:main:n}
The inequality given in the equation
\eqref{main:inequality:n}
holds for all $r\in [0,e]$, 
if we restrict  $(x_1,\ldots,x_n)$ to the
hipercube $[0,1]^n$.
\end{lemma}

\begin{proof}
Before, of the start the proof, we notice that the function
$\Upsilon(x,y)=x^{a/b}-x-y^{a/b}+y$ defined from $\mathbb{R}^2_+\to\mathbb{R}$
and for $a>b$ is concave and $\Upsilon(0,0)=\Upsilon(1,0)=\Upsilon(0,1)=\Upsilon(1,1)=0$.
Then $\Upsilon(x,y)\ge 0$ for all $(x,y)\in [0,1]\times [0,1]$.
Similarly, the function $\Upsilon_s(w,z)=\Upsilon(z,w)$ for $a<b$ is 
concave and $\Upsilon_s(w,z)\ge 0$ for all $(w,z)\in [0,1]\times [0,1]$.
Now, we proceed by induction on $n$. Let us assume that the theorem
is valid for a sequence of positive numbers $(x_1,\ldots,x_k)$ for all $k< n$. 
We note that
\begin{eqnarray}
&&\sum_{i=1}^{n}x_i^{rx_i}- x_n^{rx_1}- \sum_{i=1}^{n-1}x_i^{rx_{i+1}}
\nonumber\\
&&
\hspace{1cm}
=
\Bigg[\sum_{i=1}^{n-1}x_i^{rx_i}-x_{n-1}^{rx_1}- \sum_{i=1}^{n-2}x_i^{rx_{i+1}}\Bigg]
+\Bigg[x_{n}^{rx_{n}}+x_{n-1}^{rx_{n-1}}-x_{n}^{rx_{n-1}}-x_{n-1}^{rx_{n}}\Bigg]
\nonumber\\
&&
\hspace{1cm}
\qquad
+\Bigg[x_{n}^{rx_{n-1}}-x_{n}^{rx_{1}}-x_{n-1}^{rx_{n-1}}+x_{n-1}^{rx_{1}}\Bigg]
\nonumber\\
&&
\hspace{1cm}
:=
\mathbb{K}_1+\mathbb{K}_2+\mathbb{K}_3.
\label{eq:incuction:teo:main:n}
\end{eqnarray}
The terms $\mathbb{K}_1$ and $\mathbb{K}_2$ are positive by the inductive hypothesis.
Meanwhile, the term $\mathbb{K}_3$ is positive by the coancavity
of the functions $\Upsilon$ and $\Upsilon_s.$
Note that $a=x_{n-1}$ and $b=x_n$ and  
$\mathbb{K}_3=\Upsilon(x_{n}^{rx_1},x_{n-1}^{rx_1})$
or $\mathbb{K}_3=\Upsilon_s(x_{n-1}^{rx_1},x_{n}^{rx_1})$,
depending if $x_{n-1}>x_1$ or $x_{n-1}<x_1$, respectively. 
Then, by \eqref{eq:incuction:teo:main:n} we follow that
the Lemma is valid.
\end{proof}

\begin{conjeture}
\label{newconjeture_uno}
Let $n\in\mathbb{N}$ and $n>4$.
Then, the inequality
\eqref{main:inequality:n}
holds for all  $(x_1,\ldots,x_n)\in \mathbb{R}^n_+$ and $r\in [0,e]$.
\end{conjeture}

\begin{conjeture}
 \label{newconjeture_dos}
Let $n\in\mathbb{N}$ and $n\ge 3$.
Then, the inequality \eqref{main:raices:n}
holds  for all  $r\in [0,e]$.
\end{conjeture}

\begin{conjeture}
 \label{newconjeture_tres}
Let $n\in\mathbb{N}$ and $x_i\in ]0,1]^n$.
Then, the inequality
 \begin{eqnarray*}
 n\prod_{i=1}^{n}x_i^{rx_i}\ge \sum_{i=1}^{n}\Big(\prod_{j=1}^{n}x_j\Big)^{rx_i}
\end{eqnarray*}
holds for all  $r\in [0,e]$.
\end{conjeture}

\section*{Acknowledgement}
We acknowledge the support of ``Univesidad del B{\'\i}o-B{\'\i}o" (Chile) through
the research projects 124109 3/R, 104709 01 F/E and 
121909 GI/C.

\end{document}